\documentclass[11pt,reqno]{amsart}

\usepackage[T1]{fontenc}
\usepackage[utf8]{inputenc}
\usepackage{amsmath,amssymb,amsthm,amsfonts,mathtools}
\usepackage{enumitem}
\usepackage{hyperref}
\hypersetup{colorlinks=true,citecolor=blue,linkcolor=blue,urlcolor=blue}
\newtheorem{theorem}{Theorem}[section]
\newtheorem{lemma}[theorem]{Lemma}
\newtheorem{proposition}[theorem]{Proposition}
\newtheorem{corollary}[theorem]{Corollary}
\newtheorem{problem}[theorem]{Problem}

\theoremstyle{definition}
\newtheorem{definition}[theorem]{Definition}
\newtheorem{example}[theorem]{Example}
\theoremstyle{remark}
\newtheorem{remark}[theorem]{Remark}

\newcommand{\R}{\mathbb R}
\newcommand{\eps}{\varepsilon}
\newcommand{\W}{\mathcal W}

\newcommand{\norm}[1]{\left\lVert #1\right\rVert}
\newcommand{\diag}{\operatorname{diag}}
\newcommand{\SO}{\operatorname{SO}}

\begin{document}

\title[Frenet turns]{Frenet Turns}

\author[Boris Shapiro]{Boris Shapiro}
\address{Department of Mathematics, Stockholm University, SE-106 91 Stockholm, Sweden}
\email{shapiro@math.su.se}

\begin{abstract}
We discuss a problem posed by A.~Agrachev asking how many times a usual circle in $\mathbb R^n$ should be traversed to admit a deformation by curves with nowhere degenerating Frenet frame. It turns out that the answer depends on a specific topology which we consider.   
For the literal $C^n$ curve topology, the least number of turns of a plane circle
admitting arbitrarily small nondegenerate perturbations is
\[
        k(2)=1,\qquad k(3)=2,\qquad k(n)=1\quad(n\ge4).
\]
This jet-level problem is different from the original Frenet-control problem by Agrachev.  We
show that in the literal interpretation of Agrachev's problem  one has a simple
spherical Fenchel obstruction in all dimensions $n\ge4$.

To retain a nontrivial turn-counting problem, we introduce decorated turn data.  In
$\R^4$ the datum is a pair $(p,q)$ recording tangent-plane and normal-plane turns; we
prove that every nonresonant pair $(p,q)$ with  $p,q>0$, $p\ne q$, is accessible by small positive
constant Frenet controls.  In even dimension $2r$ the analogous datum is a vector
$(p_1,\ldots,p_r)$, and every vector with pairwise distinct positive entries is
accessible by constant controls.  Odd dimensions require genuinely time-dependent
openings since constant controls cannot close the base curve.
\end{abstract}


\maketitle

\section{Introduction}

Set $S^1=\R/2\pi\mathbb Z$.  A smooth closed curve $\gamma:S^1\to\R^n$ is called 
\emph{nondegenerate} if
\begin{equation}\label{eq:Wronskian}
        \W_\gamma(t)=\det(\gamma'(t),\gamma''(t),\ldots,\gamma^{(n)}(t))
        \ne0
\end{equation}
for all $t$.  In dimension two this means that the velocity and curvature are
nonzero.  In dimension three it is the usual nonvanishing torsion condition, provided
the lower Frenet curvatures are nonzero.

Agrachev asked for the minimal number $\mu(n)$ of turns of a plane convex curve
which should admit small regular perturbations in $\R^n$, and related this question
to periodic trajectories of the Frenet control system with positive controls.  He
recalls the classical facts
\[
        \mu(2)=1,\qquad \mu(3)=2,
\]
and asks for the answer in dimensions $n>3$, together with a lower bound for the
length of the Frenet frame \cite[Sec.~VI]{Agrachev}.  The formulation is subtle:
there are at least two natural interpretations of ``small perturbation''.  One may ask
for smallness of the curve in the $C^n$ topology, or for smallness of the Frenet
controls.  These notions agree with low-dimensional intuition but separate sharply
in dimensions $n\ge4$.

For related work on nondegenerate and locally convex curves, including the topology
of spaces of curves with prescribed Frenet frames, see
\cite{ShapiroShapiro,SaldanhaShapiro}.  Classical integral-geometric estimates for
closed curves go back to Fenchel and Milnor \cite{Fenchel,Milnor}; see also
\cite{NovikovYakovenko} for higher-dimensional integral curvature estimates.

This note has three goals.  First, we solve the literal $C^n$ curve problem.  Second,
we show that the naive fixed-normal-frame control interpretation has a simple
negative answer in dimensions $n\ge4$.  Finally, we formulate a modified version
which still counts turns and is not made trivial by this obstruction.  In dimension
four the correct boundary datum is a decorated turn vector $(p,q)$; in higher even
dimensions it becomes a vector of rotation numbers attached to a maximal torus of
$\SO(n)$.

For a positive integer $m$ put
\[
        c_m(t)=(\sin mt,\cos mt,0,\ldots,0)\in\R^n .
\]
Thus $c_m$ is the plane circle traversed $m$ times.  Our first result is the following.

\begin{theorem}\label{thm:Cn-main}
Let $k(n)$ be the least positive integer $m$ such that $c_m$ admits arbitrarily
small nondegenerate perturbations in the $C^n$ topology.  Then
\[
        k(2)=1,\qquad k(3)=2,\qquad k(n)=1\quad(n\ge4).
\]
\end{theorem}

The equality $k(3)=2$ is the classical Fenchel--Milnor phenomenon quoted by
Agrachev.  The perhaps surprising point is that, in the literal $C^n$ topology, all
higher dimensions $n\ge4$ again allow a one-turn perturbation.  This is not a
solution of the Frenet control problem.

For a unit-speed nondegenerate curve with positively oriented Frenet frame
$E=(e_1,\ldots,e_n)$, the Frenet equations have the form
\begin{equation}\label{eq:Frenet-system}
\begin{split}
        \gamma'&=e_1,\\
        e_i'&=u_i e_{i+1}-u_{i-1}e_{i-1},\qquad i=1,\ldots,n,
\end{split}
\end{equation}
where $u_0=u_n=0$ and $u_i>0$ for $1\le i\le n-1$.

\begin{definition}\label{def:Frenet-control}
The map
\[
        u=(u_1,\ldots,u_{n-1}):[0,T]\to \R_{>0}^{n-1}
\]
appearing in \eqref{eq:Frenet-system} is called the \emph{Frenet control} of
$\gamma$ with respect to the chosen orientation and arclength parameter.  Equivalently,
\[
        u_i(t)=\langle e_i'(t),e_{i+1}(t)\rangle
        =-\langle e_{i+1}'(t),e_i(t)\rangle,
        \qquad i=1,\ldots,n-1 .
\]
Thus the Frenet controls are the positive Frenet curvatures of the curve.  A
\emph{closed Frenet trajectory} means a solution of \eqref{eq:Frenet-system} for
which both the base curve and the frame close, i.e.
$\gamma(T)=\gamma(0)$ and $E(T)=E(0)$.  Its frame length is
\[
        L_F(\gamma)=\int_0^T \sqrt{u_1(t)^2+\cdots+u_{n-1}(t)^2}\,dt .
\]
\end{definition}

If a multiply covered plane curve is interpreted as the boundary control
\begin{equation}\label{eq:fixed-boundary}
        u_2=\cdots=u_{n-1}=0
\end{equation}
with a fixed normal frame, then the following obstruction applies.

\begin{theorem}\label{thm:control-obstruction-intro}
Let $n\ge4$, and let $E=(e_1,\ldots,e_n)$ be a closed positive Frenet frame
satisfying \eqref{eq:Frenet-system}.  Then
\begin{equation}\label{eq:last-controls-intro}
        \int_0^T\sqrt{u_{n-2}(t)^2+u_{n-1}(t)^2}\,dt\ge2\pi .
\end{equation}
Consequently a fixed normally framed multiply covered plane curve cannot be
approached by positive Frenet controls in any norm topology forcing the last two controls
to tend to zero in $L^1$.
\end{theorem}

Thus the fixed-normal-frame interpretation has no finite turn number in dimensions
$n\ge4$.  This is too rigid to be a satisfactory version of Agrachev's question.  In
dimension four the natural repair is to include the limiting rotation of the normal
$2$-frame.  On the interval $0\le t\le2\pi$ we write the degenerate constant boundary
control as
\begin{equation}\label{eq:turn-vector-boundary-intro}
        (u_1,u_2,u_3)=(p,0,q),\qquad p,q\in\mathbb Z_{\ge0},\quad p\ge1.
\end{equation}
The corresponding boundary frame is
\[
        R_{12}(pt)\oplus R_{34}(qt),
\]
so $p$ counts the turns of the tangent plane and $q$ counts the turns of the normal
plane.  We call $(p,q)$ a decorated turn vector.

\medskip
Our positive result is elementary but useful.

\begin{theorem}\label{thm:opening-intro}
Let $p,q$ be distinct positive integers.  Then the dimension-four boundary datum
$(u_1,u_2,u_3)=(p,0,q)$ admits openings by positive constant controls with the middle control arbitrarily small.
More precisely, for all sufficiently small $\eps>0$ there exist constants
$a_\eps,c_\eps>0$ such that the constant controls
\[
        (u_1,u_2,u_3)=(a_\eps,\eps,c_\eps)
\]
generate a closed curve with closed Frenet frame on $[0,2\pi]$, and
\[
        a_\eps\to p,
        \qquad
        c_\eps\to q .
\]
\end{theorem}

Hence the first explicit one-turn model in the four-dimensional Frenet control sense
is not the fixed-normal-frame vector $(1,0)$, but the decorated vector $(1,2)$.  This
leads to a nontrivial replacement for the naive control version of Agrachev's problem:
determine which decorated turn data lie in the closure of positive tridiagonal Frenet
trajectories.

\section{\texorpdfstring{The literal $C^n$ curve problem}{The literal Cn curve problem}}

We prove Theorem~\ref{thm:Cn-main}.  The cases $n=2$ and $n=3$ are recalled only
to fix the normalization.  The once-run plane circle is nondegenerate as a plane
curve, so $k(2)=1$.  In $\R^3$, a once-run convex plane curve has unavoidable
flattenings under sufficiently small spatial perturbations, whereas a twice-run
convex plane curve admits a perturbation with nonvanishing torsion.  This is the
classical Fenchel--Milnor phenomenon cited in \cite[Sec.~VI]{Agrachev}; see also
\cite{Fenchel,Milnor}.  Thus $k(3)=2$.

It remains to construct one-turn perturbations for all $n\ge4$.  Put $q=n-2$.

\begin{lemma}\label{lem:auxiliary-curve}
For every $q\ge2$ there exists a smooth $2\pi$-periodic curve $A:S^1\to\R^q$ such
that
\begin{equation}\label{eq:auxiliary-wronskian}
        \det(A'(t),A''(t),\ldots,A^{(q)}(t))\ne0
        \qquad(t\in S^1),
\end{equation}
and all components of $A$ have zero first Fourier harmonic.
\end{lemma}

\begin{proof}
For $q=2$ take
\[
        A(t)=(\cos2t,\sin2t).
\]
For $q=3$ take a smooth closed space curve with nonvanishing curvature and
torsion, and replace $t$ by $2t$.  This removes the first Fourier harmonic and
preserves nondegeneracy.  Such curves are classical; explicit closed examples with
constant nonzero torsion and positive curvature were constructed in \cite{BatesMelko}.

Assume the claim known in dimension $q$ and choose an integer $M\ge2$.  Define
\[
        \widetilde A(t)=(A(t),\cos Mt,\sin Mt)\in\R^{q+2}.
\]
The first harmonic is still absent.  Expanding
$\det(\widetilde A',\ldots,\widetilde A^{(q+2)})$ with respect to the last two rows,
the term in which the last two rows use the last two derivative columns is
\[
        \pm M^{2q+3}\det(A',A'',\ldots,A^{(q)}).
\]
All other terms have degree at most $2q+2$ in $M$.  The determinant in
\eqref{eq:auxiliary-wronskian} has a fixed sign and is bounded away from zero on
$S^1$.  Hence the leading term dominates uniformly for all sufficiently large $M$,
and the determinant in dimension $q+2$ is nowhere zero.
\end{proof}

Let $A$ be as in Lemma~\ref{lem:auxiliary-curve}.  Since $A$ has no first Fourier
harmonic, the equation
\begin{equation}\label{eq:solve-D2-plus-one}
        (D^2+1)Y=A
\end{equation}
has a smooth $2\pi$-periodic solution $Y:S^1\to\R^q$.  For positive parameters
$\sigma_1,\ldots,\sigma_q$ put
\[
        \Sigma=\diag(\sigma_1,\ldots,\sigma_q)
\]
and define
\begin{equation}\label{eq:lifted-curve}
        \gamma_\Sigma(t)=(\sin t,\cos t,\Sigma Y(t))\in\R^{q+2}=\R^n .
\end{equation}

\begin{proposition}\label{prop:Cn-lift}
For every $n\ge4$ and every $\delta>0$ there is a choice of positive
$\sigma_1,\ldots,\sigma_q$ such that $\gamma_\Sigma$ is closed, nondegenerate, and
\[
        \norm{\gamma_\Sigma-c_1}_{C^n}<\delta .
\]
\end{proposition}

\begin{proof}
The $C^n$ closeness is immediate once the $\sigma_j$ are sufficiently small.  Let
$C_j=\gamma_\Sigma^{(j)}$ be the $j$th derivative column in the Wronskian.  For
$j=n,n-1,\ldots,3$ replace $C_j$ by $C_j+C_{j-2}$.  These column operations do not
change the determinant.  The first two coordinates of the modified columns
$C_3,\ldots,C_n$ vanish because
\[
        (D^2+1)\sin t=(D^2+1)\cos t=0.
\]
The lower $q\times q$ block is
\[
        \Sigma A'(t),\Sigma A''(t),\ldots,\Sigma A^{(q)}(t).
\]
The determinant of the first two remaining columns in the first two coordinates is
$-1$.  Therefore
\[
        \det(\gamma_\Sigma',\gamma_\Sigma'',\ldots,\gamma_\Sigma^{(n)})
        =\pm\det(\Sigma)\det(A',A'',\ldots,A^{(q)}),
\]
which is nowhere zero.
\end{proof}

This proves Theorem~\ref{thm:Cn-main}.

\section{Why Frenet controls are different}

The passage from an $n$-jet of a curve to its Frenet controls is singular near a
plane circle.  Therefore a $C^n$-small curve perturbation need not be a small
perturbation of the Frenet controls.  The following example is the simplest warning.

\begin{example}\label{ex:dimension-four-warning}
In $\R^4$ consider
\[
        \gamma_\eps(t)=(\cos t,\sin t,\eps\cos2t,\eps\sin2t).
\]
Then
\[
        \det(\gamma_\eps',\gamma_\eps'',\gamma_\eps^{(3)},\gamma_\eps^{(4)})
        =72\eps^2,
\]
so $\gamma_\eps$ is a nondegenerate $C^4$-small perturbation of the once-run circle.
However its Frenet curvatures satisfy, as $\eps\to0$,
\[
        \kappa_1\to1,
        \qquad
        \kappa_2\to0,
        \qquad
        \kappa_3\to2.
\]
Thus the highest Frenet control is not small.  This is why Theorem~\ref{thm:Cn-main}
is not a solution of Agrachev's control problem.
\end{example}

We now record the elementary obstruction which rules out the naive strong control
interpretation.  We use the following standard consequence of spherical convexity.

\begin{lemma}[spherical Fenchel obstruction]\label{lem:spherical-obstruction}
Let $v:S^1\to S^{N-1}$ be a closed rectifiable curve and let $a:S^1\to(0,\infty)$
be integrable.  If
\[
        \int_{S^1}a(t)v(t)\,dt=0,
\]
then the spherical length of $v$ is at least $2\pi$.
\end{lemma}

\begin{proof}
The condition says that the origin lies in the convex hull of the image of $v$.
Therefore the image is not contained in any open hemisphere.  The spherical form of
Fenchel's theorem gives length at least $2\pi$; see, for example,
\cite{Fenchel,Milnor}.
\end{proof}

\begin{theorem}[control obstruction]\label{thm:strong-control-obstruction}
Let $n\ge4$, and let $E=(e_1,\ldots,e_n)$ be a closed positive Frenet frame in
$\R^n$.  Then
\begin{equation}\label{eq:last-bound}
        \int_0^T\sqrt{u_{n-2}(t)^2+u_{n-1}(t)^2}\,dt\ge2\pi .
\end{equation}
Consequently no fixed normally framed multiply traversed plane curve can be
approached by positive Frenet controls in any norm topology implying
\[
        \|u_{n-2}\|_{L^1}+\|u_{n-1}\|_{L^1}\longrightarrow0.
\]
\end{theorem}

\begin{proof}
The last Frenet equation is
\[
        e_n'=-u_{n-1}e_{n-1}.
\]
Since $e_n(T)=e_n(0)$,
\[
        \int_0^T u_{n-1}(t)e_{n-1}(t)\,dt=0.
\]
Apply Lemma~\ref{lem:spherical-obstruction} to $v=e_{n-1}$ and $a=u_{n-1}$.  Since
\[
        e_{n-1}'=-u_{n-2}e_{n-2}+u_{n-1}e_n,
\]
the spherical length of $e_{n-1}$ is exactly the left-hand side of
\eqref{eq:last-bound}.  This proves the inequality.  The final claim follows because
the boundary controls of a normally framed plane curve have $u_{n-2}=u_{n-1}=0$.
\end{proof}

\begin{remark}\label{rem:why-not-number}
Theorem~\ref{thm:strong-control-obstruction} is an obstruction, not a useful number
of turns.  With a fixed normal frame in dimensions $n\ge4$, the strong control answer
is simply that no finite number of turns suffices.  A meaningful version of
Agrachev's problem should retain a turn-counting datum while avoiding the artificial
requirement that the limiting normal frame be fixed.
\end{remark}

\section{Decorated turn vectors in dimension four}

We now describe such a turn-counting replacement in the first higher-dimensional
case.  Work on $0\le t\le2\pi$ and identify the initial Frenet frame with the identity.
In dimension four the degenerate boundary controls
\begin{equation}\label{eq:decorated-control}
        (u_1,u_2,u_3)=(p,0,q),
        \qquad p,q\in\mathbb Z_{\ge0},\quad p\ge1,
\end{equation}
produce the boundary frame
\[
        E_{p,q}(t)=R_{12}(pt)\oplus R_{34}(qt).
\]
The base curve has $p$ tangent turns, while the normal $2$-frame has $q$ turns.  We
call $(p,q)$ the decorated turn vector.

\begin{definition}\label{def:accessible-turn-vector}
A turn vector $(p,q)$ is called \emph{strongly accessible} if, for every $\eta>0$,
there exist positive controls $u_1,u_2,u_3$ on $[0,2\pi]$ satisfying the closed-frame
and closed-curve endpoint conditions
\begin{equation}\label{eq:endpoint-conditions}
        E(2\pi)=I,
        \qquad
        \int_0^{2\pi}E(t)e_1\,dt=0,
\end{equation}
and
\[
        \|u_1-p\|_{L^\infty}+\|u_2\|_{L^\infty}+\|u_3-q\|_{L^\infty}<\eta .
\]
\end{definition}

The fixed-normal-frame case is $(p,0)$.  Theorem~\ref{thm:strong-control-obstruction}
shows that no $(p,0)$ with $p\ge1$ is strongly accessible.  Thus the normal turn is
not a decoration in name only; it is forced by the last-vector obstruction.  The next
result gives the opposite direction for all nonresonant constant turn vectors.

\begin{theorem}[constant-control opening in dimension four]\label{thm:constant-control-opening}
Every turn vector $(p,q)$ with $p,q\in\mathbb Z_{>0}$ and $p\ne q$ is strongly
accessible.  More precisely, for all sufficiently small $\eps>0$ there exist positive
constants $a_\eps,c_\eps$ such that the constant controls
\[
        (u_1,u_2,u_3)=(a_\eps,\eps,c_\eps)
\]
satisfy \eqref{eq:endpoint-conditions}, and
\[
        a_\eps\to p,
        \qquad
        c_\eps\to q
        \qquad(\eps\to0).
\]
\end{theorem}

\begin{proof}
For constant controls $(a,b,c)$, write the Frenet equations as $E'=EB$, where
\[
        B=\begin{pmatrix}
        0&-a&0&0\\
        a&0&-b&0\\
        0&b&0&-c\\
        0&0&c&0
        \end{pmatrix}.
\]
The eigenvalues of $B$ are $\pm i\lambda_1,\pm i\lambda_2$, where
\begin{equation}\label{eq:frequency-relations}
        \lambda_1^2+\lambda_2^2=a^2+b^2+c^2,
        \qquad
        \lambda_1^2\lambda_2^2=a^2c^2 .
\end{equation}
We prescribe $\lambda_1=p$, $\lambda_2=q$, and set $b=\eps$.  Then $a,c$ must
satisfy
\[
        a^2+c^2=p^2+q^2-\eps^2,
        \qquad
        ac=pq .
\]
Equivalently, $a^2$ and $c^2$ are the roots of
\[
        z^2-(p^2+q^2-\eps^2)z+p^2q^2=0 .
\]
Since $p\ne q$, the discriminant is positive for all sufficiently small $\eps$, and
the two roots tend to $p^2$ and $q^2$.  This gives positive $a_\eps,c_\eps$ with the
required limits.

By the skew-symmetric spectral theorem, the frame is closed because both frequencies
are integers; in other words $\exp(2\pi B)=I$.  Since $a_\eps c_\eps>0$, the matrix
$B$ is invertible, and therefore
\[
        \int_0^{2\pi}\exp(tB)e_1\,dt
        =B^{-1}(\exp(2\pi B)-I)e_1=0.
\]
Thus the base curve is closed as well.
\end{proof}

\begin{corollary}\label{cor:first-one-turn}
The decorated turn vector $(1,2)$ is strongly accessible.  It is the smallest
nonresonant decorated vector with one tangent turn.
\end{corollary}

\begin{remark}[the resonant case]\label{rem:resonance}
The proof of Theorem~\ref{thm:constant-control-opening} fails when $p=q$.  In fact,
there is no constant-control opening with the same double frequency, because a
positive middle control splits the double eigenvalue.  This does not rule out
strong accessibility by time-dependent controls.  The resonant vectors $(p,p)$ are
therefore small concrete instances of Agrachev's one-sided endpoint problem.
\end{remark}

\section{A revised Agrachev problem}

The preceding discussion suggests replacing the fixed-normal-frame question by the
following decorated turn problem.  This is not a reformulation of Agrachev's original
number $\mu(n)$; rather, it is a nearby boundary-value problem designed to keep the
turn-counting feature while removing the artificial obstruction caused by freezing
the normal frame.

\begin{problem}[decorated Agrachev problem in dimension four]\label{prob:decorated-R4}
Determine all pairs $(p,q)\in\mathbb Z_{\ge0}^2$, $p\ge1$, for which the degenerate
Frenet trajectory
\[
        R_{12}(pt)\oplus R_{34}(qt),
        \qquad 0\le t\le2\pi,
\]
is in the strong control closure of positive closed Frenet trajectories satisfying
\eqref{eq:endpoint-conditions}.  Equivalently, determine which degenerate controls
$(p,0,q)$ can be opened to positive controls $u_1,u_2,u_3>0$.
\end{problem}

The present note proves the two first facts about this problem:
\begin{enumerate}[label=\textup{(\roman*)}]
\item the undecorated vectors $(p,0)$ are not strongly accessible;
\item all nonresonant vectors $(p,q)$ with $p,q>0$ and $p\ne q$ are strongly
accessible by constant controls.
\end{enumerate}
Thus the problem has genuine turn-counting content in both directions.  It also
separates the topological and variational parts of Agrachev's original question.  The
number of tangent turns alone is not sufficient in dimension four; the normal turn
must be part of the boundary datum.

The dimension-four statement is the first member of a higher-dimensional pattern.
We formulate it next.

\begin{definition}
For $n=2r$, a \emph{decorated turn vector} is a vector
\[
        \mathbf p=(p_1,\ldots,p_r)\in\mathbb Z_{>0}^r .
\]
It represents the degenerate boundary frame
\begin{equation}\label{eq:even-decorated-frame}
        E_{\mathbf p}(t)=R_{12}(p_1t)\oplus R_{34}(p_2t)\oplus\cdots\oplus
        R_{2r-1,2r}(p_rt),\qquad 0\le t\le2\pi .
\end{equation}
Equivalently, in the constant Frenet matrix the odd controls are
\[
        u_{2j-1}=p_j,
        \qquad j=1,\ldots,r,
\]
while the even controls $u_2,u_4,\ldots,u_{2r-2}$ vanish.  We say that $\mathbf p$
is \emph{strongly accessible} if this degenerate datum is in the $L^\infty$-closure
of positive closed Frenet controls with closed base curve and closed frame.
\end{definition}

Here and below the tridiagonal Frenet matrix associated with constant controls
$a_1,\ldots,a_{n-1}$ is
\begin{equation}\label{eq:general-tridiagonal-B}
        B(a_1,\ldots,a_{n-1})=
        \begin{pmatrix}
        0&-a_1&0&\cdots&0\\
        a_1&0&-a_2&\cdots&0\\
        0&a_2&0&\cdots&0\\
        \vdots&\vdots&\vdots&\ddots&-a_{n-1}\\
        0&0&0&a_{n-1}&0
        \end{pmatrix}.
\end{equation}
The dimension-four theorem above is the case $r=2$ of the following elementary
inverse-spectral observation.

\begin{theorem}[even-dimensional constant openings]\label{thm:even-constant-openings}
Let $n=2r$.  If
\[
        \mathbf p=(p_1,\ldots,p_r)\in\mathbb Z_{>0}^r
\]
has pairwise distinct entries, then $\mathbf p$ is strongly accessible.  More
precisely, for every sufficiently small choice of positive numbers
\[
        b_1,\ldots,b_{r-1},\qquad b_j>0,
\]
there exist positive numbers $a_1,\ldots,a_r$ with $a_j\to p_j$ as
$b_1,\ldots,b_{r-1}\to0$ such that the constant controls
\begin{equation}\label{eq:even-opening-controls}
        u_{2j-1}=a_j,
        \qquad
        u_{2j}=b_j
        \quad (j=1,\ldots,r-1),
\end{equation}
generate a closed positive Frenet curve and a closed Frenet frame on $[0,2\pi]$.
\end{theorem}

\begin{proof}
At $b_1=\cdots=b_{r-1}=0$, the matrix is block diagonal, with $2\times2$ blocks
rotating with frequencies $a_1,\ldots,a_r$.  Since the limiting frequencies
$p_1,\ldots,p_r$ are distinct, the positive eigenfrequencies of the skew-symmetric
matrix are smooth functions
\[
        \lambda_j=\lambda_j(a_1,\ldots,a_r,b_1,\ldots,b_{r-1})
\]
near the boundary point, after ordering them near $p_j$.  At the boundary,
\[
        \lambda_j(a_1,\ldots,a_r,0,\ldots,0)=a_j,
\]
so the Jacobian of $(\lambda_1,\ldots,\lambda_r)$ with respect to
$(a_1,\ldots,a_r)$ is the identity at $(p_1,\ldots,p_r,0,\ldots,0)$.  The implicit
function theorem therefore gives unique smooth positive functions $a_j=a_j(b)$,
near $p_j$, such that
\[
        \lambda_j(a(b),b)=p_j,
        \qquad j=1,\ldots,r.
\]
For these controls all frequencies of $B$ are integers; hence $\exp(2\pi B)=I$.
Because $n$ is even and all frequencies are nonzero, $B$ is invertible.  Therefore
\[
        \int_0^{2\pi}\exp(tB)e_1\,dt
        =B^{-1}(\exp(2\pi B)-I)e_1=0.
\]
Thus the Frenet frame and the base curve are both closed.
\end{proof}

\begin{corollary}\label{cor:even-one-turn}
For $n=2r$, the decorated vector
\[
        (1,2,\ldots,r)
\]
is strongly accessible.  In particular there are explicit one-tangent-turn positive
constant-control models in every even dimension.
\end{corollary}

\begin{remark}[resonances]\label{rem:higher-resonances}
The pairwise distinctness assumption is not merely a technical convenience.  At a
resonant boundary point the spectral map ceases to provide independent coordinates
for the odd controls.  In dimension four this is exactly the exceptional case
$(p,p)$ discussed in Remark~\ref{rem:resonance}.  In higher even dimensions, repeated
entries of $\mathbf p$ should be regarded as the higher-rank resonant part of the
decorated Agrachev problem.  Constant controls are not expected to resolve all
resonances; time-dependent controls are the natural next object.
\end{remark}

Odd dimensions behave differently already for constant controls.  Let $n=2r+1$ and
consider a positive constant tridiagonal Frenet matrix $B$.  Its spectrum necessarily
contains $0$.  If the nonzero frequencies are integers, then $\exp(2\pi B)=I$, but
this is not enough to close the base curve.

\begin{proposition}[odd-dimensional constant-control obstruction]
\label{prop:odd-constant-obstruction}
Let $n=2r+1$, and let $B=B(a_1,\ldots,a_{2r})$ have all $a_i>0$.  If
$\exp(TB)=I$ for some $T>0$, then
\[
        \int_0^T\exp(tB)e_1\,dt\ne0.
\]
Consequently no positive constant-control Frenet trajectory in odd dimension can
have both closed frame and closed base curve.
\end{proposition}

\begin{proof}
The kernel of $B$ is one-dimensional.  Solving $Bv=0$ along the tridiagonal chain
shows that $v_{2j}=0$ and that all odd coordinates are determined recursively from
$v_1$; in particular every nonzero vector in $\ker B$ has $v_1\ne0$.  Hence $e_1$
has nonzero orthogonal projection onto $\ker B$.  If $\exp(TB)=I$, then the integral
of $\exp(tB)e_1$ over one period is exactly $T$ times this kernel projection, and
is therefore nonzero.
\end{proof}

This proposition explains why the odd-dimensional extension of the decorated turn
approach cannot be a verbatim constant-control statement.  The natural boundary datum
is still
\[
        R_{12}(p_1t)\oplus R_{34}(p_2t)\oplus\cdots\oplus
        R_{2r-1,2r}(p_rt)\oplus 1,
\]
but a closed opening, if it exists, must use time-dependent controls to cancel the
unavoidable drift in the zero-frequency direction.

\begin{problem}[higher-dimensional decorated Agrachev problem]
\label{prob:higher-decorated}
Determine the decorated turn vectors which belong to the strong control closure of
positive closed Frenet trajectories.  More explicitly:
\begin{enumerate}[label=\textup{(\alph*)}]
\item for $n=2r$, decide whether resonant vectors, i.e. vectors with repeated
entries, can be opened by time-dependent positive controls;
\item for $n=2r+1\ge5$, decide which vectors
$(p_1,\ldots,p_r)$ admit time-dependent positive openings which close both the
frame and the base curve.
\end{enumerate}
The first expected one-tangent-turn test case in odd dimension is
\[
        (1,2,\ldots,r).
\]
The new condition in odd dimensions is not spectral closure of the frame, but
cancellation of the translational drift along the instantaneous zero-frequency
component.
\end{problem}

Thus, in the proposed higher-dimensional form, Agrachev's question asks for the
closure of positive tridiagonal Frenet trajectories near a degenerate maximal-torus
trajectory, not near a plane curve with a frozen normal frame.  This retains a genuine
number of turns: the first component $p_1$ counts the turns of the tangent plane, and
the remaining components record the necessary normal turns.  Theorem~\ref{thm:even-constant-openings}
settles the nonresonant constant-control part of this problem in even dimensions;
Problem~\ref{prob:higher-decorated} records the remaining resonance and odd-dimensional
questions.

\section{Frame length}

Agrachev also asks for lower bounds on the length of the Frenet frame.  For a
positive Frenet curve this length is
\[
        L_F(\gamma)=\int_0^T\sqrt{u_1^2+\cdots+u_{n-1}^2}\,dt .
\]
In dimension four, Fenchel's theorem applied to the tangent indicatrix gives
\[
        \int_0^T u_1(t)\,dt\ge2\pi,
\]
while Theorem~\ref{thm:strong-control-obstruction} gives
\[
        \int_0^T\sqrt{u_2(t)^2+u_3(t)^2}\,dt\ge2\pi .
\]
Consequently, by Minkowski's integral inequality,
\begin{equation}\label{eq:universal-length-bound}
        L_F(\gamma)\ge2\pi\sqrt2 .
\end{equation}
On the other hand, let $\mathcal C_{p,q}$ denote the class of positive openings of the dimension-four decorated datum $(p,q)$.  The degenerate boundary frame with dimension-four turn vector
$(p,q)$ has length
\[
        2\pi\sqrt{p^2+q^2},
\]
and the openings in Theorem~\ref{thm:constant-control-opening} have lengths tending
to this value.  In particular the first one-turn decorated model $(1,2)$ gives
\[
        \inf_{\mathcal C_{(1,2)}} L_F\le2\pi\sqrt5
\]
for openings in that decorated class.

The same upper-bound mechanism applies in even dimension.  Denote by $\mathcal C_{\mathbf p}$ the class of positive openings of the even-dimensional decorated datum $\mathbf p$.  For $n=2r$ and a
nonresonant decorated vector $\mathbf p=(p_1,\ldots,p_r)$, the constant-control
openings of Theorem~\ref{thm:even-constant-openings} have Frenet lengths tending to
\[
        2\pi\sqrt{p_1^2+\cdots+p_r^2}.
\]
Thus the model one-tangent-turn vector $(1,2,\ldots,r)$ gives the explicit bound
\[
        \inf_{\mathcal C_{(1,2,\ldots,r)}} L_F\le 2\pi\sqrt{1^2+2^2+\cdots+r^2}
        =2\pi\sqrt{\frac{r(r+1)(2r+1)}6}
\]
for the infimum over openings in that decorated class.  Closing the gap between universal lower bounds and the
best accessible decorated upper bounds is a natural variational form of Agrachev's
question.

\section{Concluding remarks}

The main lesson is that three different problems are hidden in the original ``number
of turns'' question.
\begin{enumerate}[label=\textup{(\arabic*)}]
\item The literal $C^n$ curve problem has the complete answer
\[
        k(2)=1,
        \qquad
        k(3)=2,
        \qquad
        k(n)=1\quad(n\ge4).
\]
\item The fixed-normal-frame strong control problem has no finite answer in
dimensions $n\ge4$, by the last-vector spherical Fenchel obstruction.
\item A meaningful Frenet control version should use decorated turn data.  In
$\R^4$ the relevant datum is a pair $(p,q)$, and the first explicit one-turn model is
$(1,2)$.  In even dimension $2r$, the corresponding nonresonant datum is
$\mathbf p=(p_1,\ldots,p_r)$; the model one-turn vector is $(1,2,\ldots,r)$.  Odd
dimensions require time-dependent cancellation of the zero-frequency drift.
\end{enumerate}
This decorated version is a natural finite-dimensional replacement for the
otherwise too rigid boundary-control interpretation of Agrachev's problem.  It has
obstructions, constructions, resonances, an even-dimensional inverse-spectral
opening theorem, and a remaining odd-dimensional drift-cancellation problem.

\smallskip\noindent
\emph{Acknowledgements.} The author thanks A.~Sarychev, who mentioned the problem about Frenet turns to him in Florence in 2018, and A.~Agrachev for encouragement.

\end{document}